\documentclass [11pt]{article}
\usepackage{tikz}
\usepackage{verbatim}
\usepackage{amsmath,amscd,amsthm,amssymb}
\theoremstyle{remark}

\newtheorem{thm}{{\bf Theorem}}[section] 

\newtheorem{lem}[thm]{{\bf Lemma}}
\usepackage{graphicx}
\usepackage{multirow}
\usepackage[all,pdf]{xy}
\usepackage{subfigure}
\usepackage{overpic}
\usepackage{caption}
\usepackage{CJK}
\usepackage[export]{adjustbox}
\usepackage[utf8]{inputenc}
\usepackage[T1]{fontenc}
\usepackage{booktabs}
\usetikzlibrary{calc}
\newtheorem{theorem}{Theorem}[section]

\newtheorem{conj}[theorem]{Conjecture}

\begin{document}
\begin{CJK*}{GBK}{song}
\title{Strongly regular generalized partial geometries and associated LDPC codes\footnote{Supported by  NSFC under Grant 12171139  and 11871019.}}

\author{Lijun Ma, Changli Ma, Zihong Tian\thanks{Corresponding author. E-mail address: tianzh68@163.com } \\
\small   School of Mathematical Sciences, Hebei Normal University, Shijiazhuang 050024, P. R. China}

\date{}
\maketitle
\def\binom#1#2{{#1\choose#2}}
\def \lb{\lbrack}
\def \rb{\rbrack}
\def \l{\langle}
\def \r{\rangle}
\def \m{\equiv}
\def \lb{\lbrack}
\def\rb{\rbrack}
\def\sub{\subseteq}
\def\inf{\infty} \def\om{\omega}
 \def\ba{\bigcap}
 \def\bu{\bigcup}
\def\sm{\setminus}
\def\ct{{\cal T}}
\def\ca{{\cal A}}
\def\cb{{\cal B}}
\def \o{\overline}
\def \a{\alpha}
\begin{minipage}[t]{13cm}
{\bf Abstract:}
In this paper, we introduce strongly regular generalized partial geometries of grade $r$, which generalise partial geometries and strongly regular $(\alpha,\beta)$-geometries.
By the properties of quadrics in PG$(2,q)$ and PG$(3,q)$, we construct two classes of strongly regular generalized partial geometries of grade $3$. Besides, we define low-density parity-check (LDPC) codes by considering the combinatorial structures of strongly regular generalized partial geometries and derive bounds on minimum distance, dimension and girth for the LDPC codes.

\vspace{0.3cm}
\vspace{0.3cm}

\vspace{0.3cm}

\end{minipage}

\baselineskip=17pt

\section{Introduction}
Partial geometries are members of a broad class of combinatorial configurations with geometric properties referred to as finite geometries and closely related to combinatorial design, graph theory and coding, please refer to \cite{abc1,SA,St,abc2,abc3}   for details.
In 1994, Clerck and Maldeghem \cite{abc4} generalized partial geometry to $(\alpha,\beta)$-geometry.

An \emph{$(\alpha, \beta)$-geometry} with parameters $(s,t)$ is an incidence structure $(\mathcal{P},\mathcal{B})$,  where $\mathcal{P}$ is a finite non-empty set of elements called \emph{points}, $\mathcal{B}$ is a family of subsets of $\mathcal{P}$ called \emph{blocks}, incidence being containment such that

(i) any two distinct points are incident with at most one block;

(ii) each block is incident with exactly $s+1$ points;

(iii) each point is incident with exactly $t+1$ blocks;

(iv) for any point $P\in \mathcal{P}$ and any block $B\in \mathcal{B}$ with $P\notin B$, there are exactly $\alpha$ or $\beta$ points on $B$ joined to $P$, and call $B$ an \emph{$\alpha$-block} or  \emph{$\beta$-block} with respect to $P$.

The \emph{point graph} of an incidence structure $(\mathcal{P},\mathcal{B})$ has vertex set $\mathcal{P}$ with two vertices, $P_1$ and $P_2$, connected if and only if there is a block $B\in\mathcal{B} $ such that $P_1,P_2\in B$.
A $k$-regular graph on $v$ vertices is called a \emph{strongly regular graph} with parameters $(v,k,\lambda,\mu)$ if any two adjacent vertices have $\lambda$ common neighbours and any two non-adjacent vertices have $\mu$ common neighbours.
An $(\alpha, \beta)$-geometry is called \emph{strongly regular} if its point graph is a strongly regular graph.

In 2001, Hamilton and Mathon \cite{abc5} studied strongly regular $(\alpha, \beta)$-geometries and gave some necessary conditions for existence and constructions.
Especially, a strongly regular $(0, \beta)$-geometry is a \emph{semipartial geometry} \cite{abc6}; if there exists a unique $\alpha$-block for each point, it is an \emph{almost partial geometry} \cite{abc7}.
When $\alpha=\beta$, an $(\alpha, \beta)$-geometry is a \emph{partial geometry} PG$(s,t,\alpha)$ \cite{abc1}. The point graph of a PG$(s,t,\alpha)$ is strongly regular with parameters $((s+1)(st/\alpha+1),s(t+1),s-1+t(\alpha-1),\alpha(t+1))$.
Note that the point graph of an $(\alpha, \beta)$-geometry  is not necessarily a  strongly regular graph.

In this paper, we will further extend $(\alpha,\beta)$-geometry and propose the concept of generalized partial geometry.

Throughout  this paper, we always assume that $q$ is an odd prime power.
In Section 2, we first generalise strongly regular $(\alpha,\beta)$-geometries to strongly regular generalized partial geometries of grade $r$.
In Section 3, we construct two classes of strongly regular generalized partial geometries of grade $3$ based on quadrics in PG$(2,q)$ and PG$(3,q)$.
In Section 4, we study the LDPC codes related to strongly regular generalized partial geometries.
In Section 5, we discuss the $2$-ranks of the incidence matrices of strongly regular generalized partial geometries from quadrics.
Further, we show the performance of the LDPC codes.

\section{Generalized partial geometries of grade $r$}
In this section, we first generalise strongly regular $(\alpha,\beta)$-geometries to strongly regular generalized partial geometries of grade $r$.
Then we discuss the relations of the parameters for strongly regular generalized partial geometries.

A \emph{generalized partial geometry of grade $r$}, denoted by PG$(s,t;\alpha_1,\alpha_2,\dots,\alpha_r)$, is an incidence structure $(\mathcal{P},\mathcal{B})$,  where $\mathcal{P}$ is a finite non-empty set of elements called \emph{points}, $\mathcal{B}$ is a family of subsets of $\mathcal{P}$ called \emph{blocks}, incidence being containment such that

(i) any two distinct points are incident with at most one block;

(ii) each block is incident with exactly $s+1$ points;

(iii) each point is incident with exactly $t+1$ blocks;

(iv) for any point $P\in \mathcal{P}$ and any block $B\in \mathcal{B}$ with $P\notin B$, there is $\alpha\in\{\alpha_1,\alpha_2,\dots,\alpha_r\}$ such that there are $\alpha$ points on $B$ joined to $P$, and call $B$ an \emph{$\alpha$-block} with respect to $P$;

(v) for each $\alpha\in\{\alpha_1,\alpha_2,\dots,\alpha_r\}$, there exists a point $P$ and a block $B$ with $P\notin B$ such that there are $\alpha$ points on $B$ joined to $P$.

Obviously, a generalized partial geometry of grade $2$ is an $(\alpha, \beta)$-geometry.
Suppose $v=|\mathcal{P}|$ and $n=|\mathcal{B}|$, then $v(t+1)=n(s+1)$.
We say a generalized partial geometry PG$(s,t;\alpha_1,\alpha_2,\dots,\alpha_r)$ is \emph{strongly regular} if its point graph is a strongly regular graph, denoted by SRPG$(s,t;\alpha_1,\alpha_2,\dots,\alpha_r;\lambda,\mu)$, and the parameters are
$$k=s(t+1),v=\frac{k(k-\lambda-1)}{\mu}+k+1.$$

The incidence of a generalized partial geometry $(\mathcal{P},\mathcal{B})$ can be described by a $v\times n$ \emph{incidence matrix} $M=(m_{ij})$, with rows indexed by the points, columns indexed by the blocks, and
$$m_{ij}=\left\{\begin{array}{ll}1&P_i\in B_j;\cr 0&P_i\not\in B_j, \end{array}\right.$$
where $P_i\in\mathcal{P},B_j\in\mathcal{B},0\leq i\leq v-1,0\leq j\leq n-1$.

From the properties of strongly regular graphs \cite{ab1} and the definition of SRPG$(s,t;$ $\alpha_1,\alpha_2,\dots,\alpha_r;\lambda,\mu)$, we   have the following results.
\begin{lem}\label{20}
Suppose $(\mathcal{P},\mathcal{B})$ is an SRPG$(s,t;\alpha_1,\alpha_2,\dots,\alpha_r;\lambda,\mu)$. Then the following conditions hold:

(1) $\mu~|~k(k-\lambda-1)$ and $(s+1)~|~v(t+1);$

(2) The point graph of $(\mathcal{P},\mathcal{B})$ has $3$ eigenvalues, one of which is $k$, the other two are $u_1, u_2$, where  $u_1<u_2$ with the  multiplicities $f_1, f_2$  and
$$~f_1+f_2=v-1,~u_1f_1+u_2f_2=-k, u_i^2+(\mu-\lambda)u_i-(k-\mu)=0, i=1,2.$$

\end{lem}

\begin{lem}\label{21}
Suppose $(\mathcal{P},\mathcal{B})$ is a PG$(s,t;\alpha_1,\alpha_2,\dots,\alpha_r)$ with $v=|\mathcal{P}|$. For any two points $P,Q\in\mathcal{P}$, if there are $p_i$ blocks on $P$ that are $\alpha_i$-blocks with respect to $Q$ when $P$ and $Q$ are in the same block,
$l_i$ blocks otherwise, $1\leq i\leq r$, then $(\mathcal{P},\mathcal{B})$ is strongly regular, and has parameters
$$k=s(t+1),\lambda=\sum_{i=1}^{r}(\alpha_i-1)p_i+(s-1),\mu=\sum_{i=1}^{r}\alpha_il_i$$
and $\sum_{i=1}^{r}p_i=t$, $\sum_{i=1}^{r}l_i=t+1$.
\end{lem}

Obviously, if there is a group action that transforms two adjacent points to two fixed points, and also transforms two non-adjacent points to two fixed points in a generalized partial geometry PG$(s,t;\alpha_1,\alpha_2,\dots,\alpha_r)$, then it is strongly regular.

\section{Constructions of SRPGs}
In this section, we construct two classes of strongly regular generalized partial geometries of grade $3$ based on quadrics in PG$(2,q)$ and PG$(3,q)$.
Let $\mathbb{F}_{q}$ be the finite field with $q$ elements, and PG$(m,q)$ the classical projective space over $\mathbb{F}_q$.
In PG$(m,q)$, the points are $0$-subspaces, the lines are $1$-subspaces, and the $r$-flats $(0\leq r\leq m)$ are $r$-subspaces.
Denote the set of non-zero elements of $\mathbb{F}_{q}$ by $\mathbb{F}_{q}^*$.
\begin{lem}\cite{wan2}\label{lem1}
The projective general linear  group PGL$_{m+1}(q)$ is transitive on the set of $r$-flats $(0\leq r\leq m)$ in PG$(m, q)$.
\end{lem}

The set of points $X=(x_0,x_1,\dots,x_m)$ of PG$(m,q)$ which satisfy a quadratic homogeneous equation
\begin{equation}\label{equt1}
\sum_{0\leq i\leq j\leq m}a_{ij}x_ix_j=0,
\end{equation}
where $a_{ij}\in \mathbb{F}_{q}$ and not all $a_{ij}$ are zero, $0\leq i\leq j\leq m$, is called a \emph{quadric} in PG$(m,q)$ and the quadratic equation (\ref{equt1}) is called its \emph{equation}. A quadric is called \emph{non-degenerate} if it is not reducible to a form in fewer than $m+1$ variables by a linear transformation. We use the notation $\mathcal{Q}_m$ to represent a non-degenerate quadric.

\begin{lem}\cite{abc8}\label{mj1}
For $m$ even, the equation of $\mathcal{Q}_m$ is equivalent to
\begin{equation}\label{r0}
\sum_{0\leq i\leq (m-2)/2}x_{2i}x_{2i+1}+x_m^2=0,
\end{equation}
which is called a conic $(m=2)$ or a parabolic quadric $(m>2)$.

For $m$ odd, the equation of $\mathcal{Q}_m$ is equivalent to either
\begin{equation}\label{r1}
\sum_{0\leq i\leq (m-1)/2}x_{2i}x_{2i+1}=0,
\end{equation}
which is called a hyperbolic quadric,
or
\begin{equation}\label{r2}
\sum_{0\leq i\leq (m-3)/2}x_{2i}x_{2i+1}+f(x_{m-1},x_{m})=0,
\end{equation}
which is called an elliptic quadric,
where $f(x_{m-1},x_{m})$ is an irreducible polynomial over $\mathbb{F}_{q}$.
\end{lem}

A \emph{generator} of a quadric $\mathcal{Q}_m$ is a subspace of maximum dimension in $\mathcal{Q}_m$.
In this paper, we use the generators of $\mathcal{Q}_m$ to construct generalized partial geometries.
\begin{lem}\cite{abc8}\label{w1}
The dimension of a generator of $\mathcal{Q}_m$ is $\frac{1}{2}(m-3+\omega)$, where
$$\omega=\left\{\begin{array}{ll}1,&m~\rm{is~even};\cr 2&\mathcal{Q}_m~\rm{is~hyperbolic};\cr 0&\mathcal{Q}_m~\rm{is~elliptic}. \end{array}\right.$$
\end{lem}

When $q$ is odd, we can assume that in Equation (\ref{equt1}) of a quadric the summation range is $0\leq i, j\leq m$ and $a_{ij}=a_{ji}(0\leq i, j\leq m)$.
So we could represent a quadric with the corresponding coefficient matrix $A=(a_{ij})_{0\leq i, j\leq m}$, which is an $(m+1)\times(m+1)$ symmetric matrix.
Note that $A$ and $kA,k\in \mathbb{F}_{q}^*$, represent the same quadric.
A quadric is non-degenerate if and only if its corresponding coefficient matrix is non-degenerate.

Let $A$ and $B$ be $(m+1)\times(m+1)$ matrices over $\mathbb{F}_{q}$. If there is an $(m+1)\times(m+1)$ non-degenerate matrix $Q$ over $\mathbb{F}_{q}$ such that $QAQ^T=B$, we say that $A$ is \emph{cogredient} to $B$. Clearly, the matrices cogredient to a symmetric matrix are also symmetric matrices.

From Reference \cite{wan}, when $q$ is odd, the corresponding symmetric matrix $A$ of $\mathcal{Q}_m$ is cogredient to
\begin{equation}\label{h0}
\left(\begin{array}{ccc}0&I_{\nu}&~\cr I_{\nu}&0&~\cr ~&~&1({\rm or}~z)\end{array}\right),~{\rm when}~m=2\nu,
\end{equation}
or
\begin{equation}\label{h1}
\left(\begin{array}{cc}0&I_{\nu}\cr I_{\nu}&0\end{array}\right),~{\rm when}~m+1=2\nu,
\end{equation}
or
\begin{equation}\label{h2}
\left(\begin{array}{cccc}0&I_{\nu}&~&~\cr I_{\nu}&0&~&~\cr~&~&1&~\cr~&~&~&-z\end{array}\right),~{\rm when}~m-1=2\nu,
\end{equation}
where $I_{\nu}$ is the identity matrix of order $\nu$, $z$ is a non-square element of $\mathbb{F}_{q}^*$, and matrices (\ref{h0}), (\ref{h1}) and (\ref{h2}) are corresponding to Equations (\ref{r0}), (\ref{r1}) and (\ref{r2}) respectively.

\subsection{SRPGs from conics in PG$(2,q)$}
In this section, we use conics in PG$(2,q)$ to construct strongly regular generalized partial geometries.
From Lemma \ref{w1}, the generator of a conic is a point.
Let $V$ be the set of all points in PG$(2,q)$, then $|V|=q^{2}+q+1$. For convenience, we denote a point in PG$(2,q)$ by the vector whose first nonzero position is $1$ in a $0$-subspace hereinafter.

When $q$ is an odd prime power, the equation of a conic  in PG$(2, q)$ can be represented by
$$XAX^{T}=0,$$
where $X=(x,y,z)$ is a point, $A\in \mathbb{F}_q^{3\times 3},~A=A^{T}$ and $|A|\not=0$. So we could represent a conic with the corresponding non-degenerate symmetric matrix.
From Lemma \ref{mj1}, the equation of a conic can be carried  by  a  projective  transformation into the normal form
$xy+z^2=0$ in PG$(2, q)$.

\begin{lem}\cite{abc8,wan}\label{five}
In PG$(2,q)$,

$(1)$ there are $q+1$ points in a conic and any three of them are non-collinear;

$(2)$ through any five points, no three of which are collinear, passes a unique conic, where $q\geq 4$;

$(3)$ the number of conics is $q^2(q^3-1).$
\end{lem}

For any point $P$ in PG$(2,q)$, which is different from $e_1=(1,0,0),e_2=(0,1,0),e_3=(0,0,1)$, we call that $P$ is of type $I$, if any three distinct points in the set $\{P,e_1,e_2,e_3\}$  are not collinear, type $II$ otherwise.
Let $\mathcal{P}_1$ be the set of points of type $I$
and $\mathcal{C}_1$ be the set of all the conics containing $e_1,e_2,e_3$. Then $\mathcal{P}_1=\{(1,x,y): x,y\in\mathbb{F}_q^*\}$ and the corresponding symmetric matric $A$ of the conic in $\mathcal{C}_1$ has the form
\begin{equation}\label{equa}
A=\left(\begin{array}{ccc}0&a&b\cr a&0&1\cr b&1&0\end{array}\right),
\end{equation}
where $a,b\in\mathbb{F}_q^*$.
Then $|\mathcal{C}_1|=(q-1)^{2}$.
Let the incident matrix of $(\mathcal{P}_1, \mathcal{C}_1)$ be $M_1$. Then $M_{1}$ is a $(q-1)^{2}\times (q-1)^{2}$ matrix over  $\mathbb{F}_2$.

\begin{thm}\label{m23}
For the incidence matrix $M_1$, we have

$(1)$ the number of $1'$s in each column is $q-2;$

$(2)$ the number of $1'$s in each row is $q-2.$
\end{thm}
\begin{proof}
Let $\mathcal{O}$ be a conic in $\mathcal{C}_1$ and its corresponding symmetric matric $A$ be the same as Equation (\ref{equa}).
Set $P=(1,x,y)$, $x,y\in\mathbb{F}_q^*$. If $P\in\mathcal{O}$, there is
$$\left(\begin{array}{ccc}1&x&y\end{array}\right)\left(\begin{array}{ccc}0&a&b\cr a&0&1\cr b&1&0\end{array}\right)\left(\begin{array}{c}1\cr x\cr y\end{array}\right)=2ax+2by+2xy=0,$$
then $y=-ax(b+x)^{-1}$ and $x\neq -b$ as $A$ is a non-degenerate symmetric matrix.
Thus, the number of points in $\mathcal{O}$ is $q-2$, that is (1).

It is well known that the group PGL$_{3}(q)$  is $4$-transitive on the set of points in which any three points are non-collinear, so the number of $1'$s in each row is the same, which is also $q-2$.
\end{proof}

Especially, when $q=3$, $|\mathcal{P}_1|=|\mathcal{C}_1|=4$ and $(\mathcal{P}_1,\mathcal{C}_1)$ only has four isolated points. Then it is an SRPG$(0,0;0;0,0)$. Next we discuss the case of $q\geq5$.

\begin{thm}\label{ld11}
For $q\geq5$, the point graph of $(\mathcal{P}_1, \mathcal{C}_1)$ is a strongly regular graph with parameters
$$v=(q-1)^2,~k=(q-2)(q-3),~\lambda=(q-4)^2+1,~\mu=(q-3)(q-4).$$

\end{thm}
\begin{proof}
Denote the point graph of $(\mathcal{P}_1, \mathcal{C}_1)$ by $G_1$, then $v=|\mathcal{P}_1|=(q-1)^2$. For any two points $P,Q\in\mathcal{P}_1$, $P\sim Q$ if and only if any three of $e_1,e_2,e_3,P,Q$ are non-collinear.
For any point $P\in\mathcal{P}_1$, there are $q-2$ conics containing $P$ in $\mathcal{C}_1$, and any two of them have a unique common point $P$. So there are $(q-2)(q-3)$ points adjacent to $P$ in $G_1$, that is $k=(q-2)(q-3)$.
For any two distinct points $P,Q\in\mathcal{P}_1$, without loss of generality, we can set $P=(1,1,1)$.

If $P\sim Q$, then $Q=(1,x,y)$, where $x\neq 1,y\neq 1,x$ and $x,y\in \mathbb{F}_q^*$.
Suppose $R=(1,x',y')$ is the common neighbour of $P,Q$.
By  careful  calculations, we know that $R$ is the common neighbour of $P,Q$ if and only if $x'\neq1,x,~y'\neq1,y,x',x^{-1}x'y$ and $x,y\in\mathbb{F}_q^*$.
Then we obtain a table as following according to the cases of $y'$.\vspace{0cm}

\begin{center}
\scalebox{0.9}{
\begin{tabular}{|c|c|c|c|c|}
\hline
$Q$ & $R$  & Conditions& Number& $\lambda$\\ \hline
                            &  $(1,y,y')$ & $y'\neq1,y$ & $q-3$ &\\ \cline{2-4}
$(1,y^2,y),y\neq \pm1$      &  $(1,x',y'),$ & $x'\neq1,y,y^2$; & \multirow{2}{*}{$(q-4)(q-5)$} &$(q-4)^2+1$\\
                            &  $x'\neq y$ & $y'\neq1,y,x',y^{-1}x'$ &  &\\ \hline
                            & $(1,y^{-1}x,y')$ & $y'\neq1,y,y^{-1}x$ & $q-4$ & \\ \cline{2-4}
\multirow{2}{*}{$(1,x,y),x\neq1,y,y^2;y\neq1$} & $(1,y,y')$ & $y'\neq1,y,x^{-1}y^2$ & $q-4$ & \multirow{2}{*}{$(q-4)^2+1$}  \\ \cline{2-4}
                            & $(1,x',y'),$ & $x'\neq1,x,y^{-1}x,y;$ & \multirow{2}{*}{$(q-5)^2$}& \\
                            & $x'\neq y^{-1}x,y$ & $y'\neq1,y,x',x^{-1}x'y$ &  & \\ \hline
\end{tabular}}
\end{center}
Therefore, when $P\sim Q$, the number of the common neighbours of $P,Q$ is a constant, that is $\lambda=(q-4)^2+1.$

If $P\nsim Q$, then $Q=(1,1,x),(1,x,1)$ or $(1,x,x)$, where $x\neq1$ and $x\in \mathbb{F}_q^*$. Suppose $R=(1,x',y')$ is the common neighbour of $P,Q$.  By  careful  calculations, we obtain  a table as following according to the cases of $y'$.\\ \vspace{-0.5cm}

\begin{center}
\scalebox{0.79}{
\begin{tabular}{|c|c|c|c|c|}
\hline
$Q$ & $R$  & Conditions& Number& $\mu$\\ \hline
\multirow{2}{*}{$(1,1,-1)$}  &  $(1,-1,y')$ & $y'\neq\pm1$ & $q-3$ &  \multirow{2}{*}{$(q-3)(q-4)$}\\ \cline{2-4}
            &  $(1,x',y'),~x'\neq \pm1$ & $x'\neq\pm1;~y'\neq\pm1,\pm x'$ & $(q-3)(q-5)$ &\\ \hline

                       & $(1,x^{-1},y')$ & $y'\neq1,x,x^{-1}$ & $q-4$ & \\ \cline{2-4}
$(1,1,x),x\neq\pm1$    & $(1,x,y')$ & $y'\neq1,x,x^{2}$ & $q-4$ & $(q-3)(q-4)$\\ \cline{2-4}
                       & $(1,x',y'),~x'\neq x,x^{-1}$ & $x'\neq1,x,x^{-1};~y'\neq1,x,x',xx'$ & $(q-4)(q-5)$ & \\  \hline
\multirow{2}{*}{$(1,-1,1)$}  &  $(1,x',-1)$ & $x'\neq\pm1$ & $q-3$ &  \multirow{2}{*}{$(q-3)(q-4)$}\\ \cline{2-4}
            &  $(1,x',y'),~y'\neq \pm1$ & $y'\neq\pm1;~x'\neq\pm1,\pm y'$ & $(q-3)(q-5)$ &\\ \hline

                       & $(1,x',x^{-1})$ & $x'\neq1,x,x^{-1}$ & $q-4$ & \\ \cline{2-4}
$(1,x,1),x\neq\pm1$    & $(1,x',x)$ & $x'\neq1,x,x^{2}$ & $q-4$ & $(q-3)(q-4)$\\ \cline{2-4}
                       & $(1,x',y'),~y'\neq x,x^{-1}$ & $y'\neq1,x,x^{-1};~x'\neq1,x,y',xy'$ & $(q-4)(q-5)$ & \\  \hline
$(1,x,x),x\neq1$    & $(1,x',y')$ & $x'\neq1,x;~y'\neq 1,x,x'$ & $(q-3)(q-4)$ & $(q-3)(q-4)$\\ \hline
\end{tabular}}
\end{center}
Therefore, when $P\nsim Q$, the number of the common neighbours of $P,Q$ is also a constant, that is $\mu=(q-3)(q-4)$.

To sum up, the point graph of $(\mathcal{P}_1, \mathcal{C}_1)$ is a strongly regular graph.
\end{proof}

\begin{thm}\label{partgeo}
For $q\geq5$, $(\mathcal{P}_1, \mathcal{C}_1)$ is a PG$(q-3,q-3;q-5,q-4,q-3)$.
\end{thm}
\begin{proof}
From Theorems \ref{five} and \ref{m23}, we can get the axioms (i)-(iii) of generalized partial geometry.

For axioms (iv)-(v), take any point $P\in\mathcal{P}_1$ and any conic $C_A\in\mathcal{C}_1$ such that $P\notin C_A$, where $A$ is the corresponding symmetric matrix of the conic $C_A$.
Without loss of generality, we can set $P=(1,1,1)$. Let the form of $A$ be the same as Equation (\ref{equa}), then $P\notin C_A$ if and only if $a\neq -b-1$.
By calculation, we  know that the points in $C_A$ are $(1,\frac{-by}{a+y},y)$, where $y\neq-a,y\in\mathbb{F}_q^*$. Next, we calculate the number of points joined to $P$ in $C_A$.

Suppose $Q=(1,\frac{-by}{a+y},y)\in C_A$, then $P$ and $Q$ are adjacent if and only if any three points of $e_1,e_2,e_3,P,Q$ are non-collinear. Then $y\neq1,-a,-b-a,\frac{-a}{1+b}$ and there is the following table:\\ \vspace{-0.5cm}

{\small \begin{center}
\scalebox{1}{
\begin{tabular}{|c|c|c|c|}
\hline
$b$ & $a$  & Conditions of $y$ & Number \\ \hline
\multirow{2}{*}{$b=1$}  & $a=-1$ & $y\neq1,\frac{1}{2}$ & $q-3$  \\ \cline{2-4}
       & $a\neq-1$ & $y\neq1,-a,-1-a,\frac{-a}{2}$ & $q-5$ \\ \hline

             & $a=-1$ & $y\neq1,2$ & $q-3$  \\ \cline{2-4}
$b=-1$       & $a=1$ & $y\neq1,-1$ & $q-3$ \\ \cline{2-4}
             & $a\neq\pm1$ & $y\neq1,-a,1-a$ & $q-4$ \\  \hline

              & $a=-1$ & $y\neq1,-b+1,\frac{1}{1+b}$ & $q-4$ \\ \cline{2-4}
$b\neq\pm1$   & $a=-b$ & $y\neq1,b,\frac{b}{1+b}$ & $q-4$ \\ \cline{2-4}
              & $a\neq-1,-b$ & $y\neq1,-a,-b-a,\frac{-a}{1+b}$ & $q-5$ \\  \hline
\end{tabular}}
\end{center}}
Therefore, the number of points in $C_A$ joined to $P$ is $q-5,q-4$ or $q-3$, and $(\mathcal{P}_1, \mathcal{C}_1)$ is a PG$(q-3,q-3;q-5,q-4,q-3)$.
\end{proof}

\subsection{SRPGs from hyperbolic quadrics in PG$(3,q)$}
In this section, we use hyperbolic quadrics in PG$(3,q)$ to construct strongly regular generalized partial geometries.
From Lemma \ref{w1}, the generator of a hyperbolic quadric is a line.

Let $\mathcal{L}$ be the set of all lines and $\mathcal{H}$ be the set of all hyperbolic quadrics in PG$(3,q)$, then $|\mathcal{L}|=(q^{2}+1)(q^{2}+q+1)$ and $|\mathcal{H}|=\frac{1}{2}q^{4}(q^{2}+1)(q^{3}-1)$ from Reference \cite{abc8}.

\begin{lem}\label{th12}\cite{skew}
In PG$(3, q)$, there is a unique hyperbolic quadric containing fixed three lines in which any two  lines of them are skew.
\end{lem}

The matrix representation for a line in PG$(3,q)$ under scalar multiplication is one of the following forms:
$$\left(\begin{array}{cccc}x&y&1&0\cr z&m&0&1\end{array}\right),
\left(\begin{array}{cccc}x&1&0& 0\cr z&0& a&1\end{array}\right),\left(\begin{array}{cccc}1&0&0&0\cr 0&m&a&1 \end{array}\right),$$
$$\left(\begin{array}{cccc}x&1&0&0\cr z&0&1&0\end{array}\right),\left(\begin{array}{cccc}1 &0&0&0\cr 0&m&1&0\end{array}\right),\left(\begin{array}{cccc}1&0&0 &0\cr 0&1&0&0\end{array}\right),$$
where $x,y,z,m,a\in \mathbb{F}_{q}$.
Next, we represent a line with the corresponding matrix $L$.

When $q$ is an odd prime power, the equation of a hyperbolic quadric in PG$(3, q)$ can be represented by
$$XHX^{T}=0,$$
where $X=(x,y,z,w)$ is a point, $H\in \mathbb{F}_q^{4\times 4}$ and $H=H^{T}$ is cogredient to
$$K=\left(\begin{array}{cc}0&I_{2}\cr I_{2}&0\end{array}\right). $$
In the following, we divide $H$ into $2\times2$ matrices for convenience, that is
\begin{equation}\label{equa4}
H=\left(\begin{array}{cc}A&B\cr B^{T}&C\end{array}\right),
\end{equation}
where $A,B,C\in F_q^{2\times2}$, and $A$ and $C$ are symmetric matrices.
We represent a hyperbolic quadric with the corresponding matrix $H$.

\begin{lem}\label{th13}
In the incidence structure $(\mathcal{L}, \mathcal{H})$, suppose the map $Q^\sigma$ satisfies that $Q^\sigma(L)=LQ$, $Q^\sigma(H)=Q^{-1}H(Q^{-1})^T$, where $L\in\mathcal{L},H\in\mathcal{H}$ and $Q\in GL_4(F_q)$. Then $Q^\sigma$ is isomorphic.
\end{lem}
\begin{proof}
It is easy to see that $L\in H$ if and only if $Q^\sigma(L)\in Q^\sigma(H)$, so  $Q^\sigma$ is an isomorphic map of $(\mathcal{L}, \mathcal{H})$.
\end{proof}

\begin{lem}
The non-degenerate quadrics containing the line $(I_2~0)$ are all hyperbolic.
\end{lem}
\begin{proof}
Suppose $H$ is a non-degenerate quadric   with the form as (\ref{equa4}).
Then the line $(I_2~0)$ is contained in $H$ if and only if $A=0$ and $B$ is non-degenerate.
Let
$$Q=\left(\begin{array}{cc}I_2&0\cr \frac{1}{2}C B^{-1}&B^{T}\end{array}\right)\in PGL_{4}(q),$$
then $$Q\left(\begin{array}{cc}0& I_2\cr I_2&0\end{array}\right)Q^T=H.$$
So the non-degenerate quadric containing the line $(I_2~0)$ is hyperbolic.
\end{proof}

For a given line $L$, we take all the line skew to it and the hyperbolic quadrics containing it, denoted by $\mathcal{L}_1$ and $\mathcal{H}_1$ respectively.
By the transitivity of lines, suppose $L$ is $(I_2~0)$, then $\mathcal{L}_1$ is the set of the lines with the form $(N~I_2)$, $N\in \mathbb{F}_q^{2\times2}$ and $\mathcal{H}_1$ is the set of the hyperbolic quadrics containing the line $(I_2~0)$.
Then   $\mathcal{L}_1$ and $\mathcal{H}_1$ can be represented as follows:
$$\mathcal{L}_1=\{L_N=(N~I_2):N\in \mathbb{F}_q^{2\times2}\},$$
$$\mathcal{H}_1=\{H_{B,C}=\{L_N\in\mathcal{L}_1:L_N\left(\begin{array}{cc}0&B\cr B^T&C\end{array}\right)L_N^T=0\}:B\in GL_2(\mathbb{F}_q),C^T=C\}.$$
It is not difficult to check that a quadric with above form is hyperbolic and
a line $L_N$ is contained in a hyperbolic quadric $H_{BC}$ if and only if $B^TN^T+NB+C=0$.

Let the incident matrix of $(\mathcal{L}_1, \mathcal{H}_1)$ be $M_2$.
Then $M_2$ is a $q^4\times q^4(q^2-1)$ matrix.

\begin{thm}\label{mat2}
For the incidence matrix $M_2$, we have

$(1)$ the number of $1'$s in each column is $q;$

$(2)$ the number of $1'$s in each row is $q(q^2-1).$
\end{thm}
\begin{proof}
Because the quadric $H$ containing the line $(I_{2}~0)$ is cogredient to $K$,
$L_N$ is in $H_{I_{2},0}$ if and only if $N^T+N=0$. So the number of $1'$s in each column of $M_2$ is $q$.
By the transitivity of lines, the line skew to $(I_2~0)$ can be transformed into $L_0$. So the number of $1'$s in each row  of $M_2$ is the same, that is $q(q^2-1)$.
\end{proof}

\begin{thm}\label{th31}
In the point graph of $(\mathcal{L}_1, \mathcal{H}_1)$, $L_{N_1}$ and $L_{N_2}$ are adjacent if and only if $rank(N_2-N_1)=2$.
\end{thm}
\begin{proof}
In the point graph of $(\mathcal{L}_1, \mathcal{H}_1)$, $L_{N_1}\sim L_{N_2}$ if and only if there is {\small$H=\left(\begin{array}{cc}0&B\cr B^T&C\end{array}\right)$ }such that $B^TN_i^T+N_iB+C=0,~i=1,2.$
Under the action of {\small$\left(\begin{array}{cc}I_2&0\cr -N_1&I_2\end{array}\right)$}, the two lines $L_{N_1}=(N_1~I_2)$ and $L_{N_2}=(N_2~I_2)$ can be transformed into $(0~I_2)$ and $(N_2-N_1~I_2)$.
Then a hyperbolic quadric $H$ containing $(0~I_2)$ implies $C=0$.

Suppose that $N_3=N_2-N_1$, $H$ containing $(N_3~I_2)$ implies $B^TN_3^T+N_3B=0$, that is $N_3B$ is an alternating matrix. It is well known that the rank of an alternating matrix is even, so $rank(N_3)=2$.

When $rank(N_3)=2$, under the action of {\small$\left(\begin{array}{cc}N_3^{-1}&0\cr 0&I_2\end{array}\right)$}, $(N_3~I_2)$ can be transformed into $(I_2~I_2)$. Then a hyperbolic quadric $H$ containing $(I_2~I_2)$ means $B^T+B=0$. So the hyperbolic quadric is
$$\left(\begin{array}{cc}0&B\cr B^T&0\end{array}\right),$$
where {\small$B=\left(\begin{array}{cc}0&1\cr -1&0\end{array}\right)$}.
Thus $(N_1~I_2)\sim(N_2~I_2)$ if and only if $rank(N_2-N_1)=2$.
\end{proof}

Obviously, the point graph of $(\mathcal{L}_1, \mathcal{H}_1)$ is isomorphic to the graph $G_2$ whose vertex set is all $2\times2$ matrices over $\mathbb{F}_q$ and any two vertices is adjacent if and only if the rank of their difference is $2$.

By Theorem 2.9 and Theorem 2.11 in \cite{Sche}, we can know that $G_2$ is strongly regular and its parameters are as follows.

\begin{thm}\cite{Sche}\label{ld12}
$G_2$ is a strongly regular graph with parameters
$$v=q^4,~k=q(q-1)(q^2-1),~\lambda=q(q^3-2q^2-q+3),~\mu=q(q-1)(q^2-q-1).$$
\end{thm}

\begin{thm}\label{partgeo2}
$(\mathcal{L}_1, \mathcal{H}_1)$ is a PG$(q-1,q(q^2-1)-1;q-2,q-1,q)$.
\end{thm}
\begin{proof}
From Lemma \ref{th12}, Theorems \ref{mat2} and \ref{th31}, we can get the axioms (i)-(iii) of generalized partial geometry.

For axioms (iv)-(v), take any line $L\in\mathcal{L}_1$ and any hyperbolic quadric $H\in\mathcal{H}_1$ such that $L\notin H$.
Without loss of generality, we can set $L=L_{0}=(0~I_2)$ and {\small$H=H_{B,C}=\left(\begin{array}{cc}0&B\cr B^T&C\end{array}\right)$} with $C\neq0$. From Lemma \ref{th13}, under the map of ${T_2}^\sigma$ with {\small$T_2=\left(\begin{array}{cc}B&0\cr 0&I_2\end{array}\right)$}, $L_{0}$ and $H_{B,C}$ can be transformed  into
$$L_{0},~H_{I_2,C_1}=\left(\begin{array}{cc}0&I_2\cr I_2&(B^{-1})^TCB^{-1}\end{array}\right).$$
So $L_{0}$ is not in $H_{I_2,C_1}$ if and only if $C_1\neq0$, where $C_1=(B^{-1})^TCB^{-1}$.

Next we determine the number of lines adjacent to $L_{0}$ in $H_{I_2,C_1}$. Because the line $L_{N}$ is in $H_{I_2,C_1}$ if and only if $N^T+N+C_1=0$, from Theorem \ref{th31}, we only need to calculate the number of non-degenerate matrices in the set $\{N:N^T+N+C_1=0,rank(N)=2\}$.

By the transitivity of the matrices, $C_1$ is equivalent to the matrix {\small$\left(\begin{array}{cc}1&0\cr 0&0\end{array}\right)$}, {\small$\left(\begin{array}{cc}1&0\cr 0&1\end{array}\right)$} or {\small$\left(\begin{array}{cc}1&0\cr 0&z\end{array}\right)$}, where $z$ is a non-square element of $\mathbb{F}_q$. Then we can get the number of lines adjacent to $L_{0}$ is $q-2,q-1$ or $q$ by calculation.
Therefore, $(\mathcal{L}_1, \mathcal{H}_1)$ is a PG$(q-1,q(q^2-1)-1;q-2,q-1,q)$.
\end{proof}

\section{LDPC codes from SRPGs}

Let $M$ be the incidence matrix of an SRPG$(s,t;\alpha_1,\alpha_2,\dots,\alpha_r;\lambda,\mu)$ $(\mathcal{P},\mathcal{B})$ given in Section 2, and $C(M)$ be an LDPC code with parity-check matrix $M$. In this section, we consider the LDPC code $C(M)$ and give bounds on minimum distance, dimension and girth for $C(M)$.

LDPC codes were introduced along with an iterative probability-based decoding algorithm by Gallager \cite{gal} in the early 1960's.
An LDPC code is a binary linear code defined by a sparse parity-check matrix $H$, which means that $H$ contains a very small number of nonzero entries.
An LDPC code is called $(\omega_{col},\omega_{row})$-regular if $H$ has constant column weight $\omega_{col}$ and constant row weight $\omega_{row}$, and the number of $1$'s in common between any two columns is no greater than $1$.
LDPC codes in this paper refer to regular.

Johnson et al. \cite{St} presented regular LDPC codes from partial geometries. Li et al. \cite{sta,ssp} studied regular LDPC codes from strongly regular $(\alpha, \beta)$-geometries and semipartial geometry, respectively.

By the axioms (i)-(iii) of generalized partial geometry, we know the code $C(M)$ with parity-check matrix $M$ is an $(s+1,t+1)$-regular LDPC code, and the dimension of the code is $n-$rank$_2(M)$.

\subsection{Minimum distance}
Let $H$ be a parity-check matrix of a regular code $C$ with the multiplicity of the largest eigenvalue $a_1$, of $HH^T$, is $1$.
Let $\omega_{col}$ be the column weight of $H$, $\omega_{row}$ be the row weight of $H$, and $a_2$ be the second largest distinct eigenvalue of $HH^T$.
For the minimum distance $d$ of $C$, Tanner \cite{Tan} presented the  bit-oriented bound
\begin{equation}\label{ldpc1}
d\geq\frac{n(2\omega_{col}-a_2)}{(a_1-a_2)},
\end{equation}
and parity-oriented bound \begin{equation}\label{ldpc2}
d\geq\frac{2n(2\omega_{col}+\omega_{row}-2-a_2)}{\omega_{row}(a_1-a_2)}.
\end{equation}

In a similar manner to \cite{abc8}, \cite{St} and \cite{sta}, we will use the bit- and parity-oriented bounds, together with the properties of strongly regular graphs, to derive lower bounds on $d$ in terms of $s,t,\lambda$ and $\mu$ for the code $C(M)$.
Next, we discuss the eigenvalues of $MM^T$ based on the relationship between the incidence matrix and the adjacency matrix of an SRPG$(s,t;\alpha_1,\alpha_2,\dots,\alpha_r;$ $\lambda,\mu)$.

The \emph{adjacency matrix} of an SRPG$(s,t;\alpha_1,\alpha_2,\dots,\alpha_r;$ $\lambda,\mu)$ $(\mathcal{P},\mathcal{B})$ is a square matrix $A=(a_{ij})$, indexed by the points of $\mathcal{P}$, defined as
$$a_{ij}=\left\{\begin{array}{ll}1,&  P_i,P_j\in B,i\neq j;\cr 0,&\rm{otherwise}, \end{array}\right.$$
where $P_i,P_j\in\mathcal{P}, B\in\mathcal{B}$, $0\leq i,j\leq v-1$.
From the definition of an SRPG$(s,t;\alpha_1,$ $\alpha_2,\dots,\alpha_r;\lambda,\mu)$, we know $A=MM^T-(t+1)I$.

\begin{lem}\cite{SA}\label{012}
Let $G$ be a strongly regular graph with parameters $(v,k,\lambda,\mu)$ and $A$ be the adjacency matrix of $G$. Then the eigenvalues of $A$ are $k$, $u_1$ and $u_2$, where
$$u_1,u_2=\frac{(\lambda-\mu)\pm\sqrt{\Delta}}{2}.$$
Further, the multiplicities of $k$, $u_1$ and $u_2$ are $1$, $f_1$ and $f_2$ respectively, where
$$f_1,f_2=\frac{1}{2}\big[(v-1)\pm\frac{(v-1)(\mu-\lambda)-2k}{\sqrt{\Delta}}\big]$$
with $\Delta=(\lambda-\mu)^2+4(k-\mu)$.
\end{lem}

From Lemma \ref{012}, for the incidence matrix $M$ of an SRPG$(s,t;$ $\alpha_1,\alpha_2,\dots,\alpha_r;\lambda,\mu)$, $MM^T$ has eigenvalues
$$(s+1)(t+1),\frac{(\lambda-\mu)\pm\sqrt{\Delta}}{2}+t+1$$
with multiplicities
$$1,\frac{1}{2}\big[(v-1)\pm\frac{(v-1)(\mu-\lambda)-2k}{\sqrt{\Delta}}\big].$$
Then we can derive the following result by Inequations (\ref{ldpc1}) and (\ref{ldpc2}).

\begin{thm}\label{mim}
The minimum distance $d$ of an LDPC code $C(M)$ from an  SRPG$(s,t;\alpha_1,$ $\alpha_2,\dots,\alpha_r;\lambda,\mu)$ satisfies
$$d\geq \max\big\{\frac{n(4s-2t+2-\lambda+\mu-\sqrt{\Delta})}{2s(t+1)-\lambda+\mu-\sqrt{\Delta})},\frac{2n(4s-\lambda+\mu-\sqrt{\Delta})}{(t+1)(2s(t+1)-\lambda+\mu-\sqrt{\Delta})}\big\},$$
where $n=v(t+1)/(s+1)$.
\end{thm}

\subsection{$2$-rank of $M$}

In this subsection, we use a result of Brouwer \cite{ab1} on the $p$-rank of the adjacency matrix of strongly regular graph
to discuss the $2$-rank of $M$.
\begin{lem}\cite{ab1}\label{23}
If $A$ is the adjacency matrix of a strongly regular graph with eigenvalues $k$, $u_1$, $u_2$ with multiplicities $1$, $f_1$, $f_2$ respectively, and the matrix $N$ is defined as $N=A+bJ+cI$  for some $b$ and $c$, then $N$ has eigenvalues $\theta_0=k+bv+c$, $\theta_1=u_1+c$, $\theta_2=u_2+c$ with multiplicities $1$, $f_1$, $f_2$ respectively. Further,

$(1)$ if none eigenvalue $\theta_i$ of $N$ is $\equiv0($mod $ p)$, $i=0,1,2$, then $rank_p(N)=v$;

$(2)$ if precisely one eigenvalue $\theta_i$ of $N$ is $\equiv0($mod $ p)$, then $rank_p(N)=v-m_i$, where $m_i$ is the multiplicity of the eigenvalue;

$(3)$ if $\theta_0\equiv\theta_1\equiv0($mod $ p)$ and $\theta_2\not\equiv0($mod $ p)$, then $rank_p(N)=f_2$ if $p|e$, and $rank_p(N)=f_2+1$ otherwise.
Similarly, if $\theta_0\equiv\theta_2\equiv0($mod $p)$ and $\theta_1\not\equiv0($mod $p)$, then $rank_p(N)=f_1$ if $p|e$, and $rank_p(N)=f_1+1$ otherwise;

$(4)$ if $\theta_1\equiv\theta_2\equiv0($mod $ p)$, then $rank_p(N)\leq \min\{f_1+1,f_2+1\}$.

In the above, $e=\mu+b^2v+2bk+b(\mu-\lambda)$.
\end{lem}
For $M$, the incidence matrix of an SRPG$(s,t;\alpha_1,\alpha_2,\dots,\alpha_r;\lambda,\mu)$, we can define $N=MM^T=A+(t+1)I$, and thus $k=s(t+1)$, $u_1,u_2=\frac{(\lambda-\mu)\pm\sqrt{\Delta}}{2}$, $\theta_0=(s+1)(t+1)$, $\theta_1,\theta_2=\frac{(\lambda-\mu)\pm\sqrt{\Delta}}{2}+t+1$ and $e=\mu$.
Obviously, $\lambda-\mu\equiv\theta_1+\theta_2($mod $2)$.
We may obtain $rank_2(MM^T)$ by Lemmas \ref{012} and \ref{23}.

When $\theta_1+\theta_2\equiv0 ($mod $2)$, from Lemma \ref{23}, we can know $rank_2(MM^T)$ in the case of $\theta_1,\theta_2\equiv1($mod $2)$, i.e.,
$$rank_2(MM^T)=v$$
if $\theta_0\equiv1($mod $2)$, or
$$rank_2(MM^T)=v-1$$
if $\theta_0\equiv0($mod $2)$.

When $\theta_1+\theta_2\equiv1($mod $2)$, they are classified as follows.

$(1)$ $\theta_1\equiv0($mod $2)$, $\theta_2\equiv1($mod $2)$. If $\theta_0\equiv0($mod $2)$, then
$$rank_2(MM^T)=\frac{1}{2}\big[(v-1)-\frac{(v-1)(\mu-\lambda)-2k}{\sqrt{\Delta}}\big]$$
when $2|\mu$, or
$$rank_2(MM^T)=\frac{1}{2}\big[(v-1)-\frac{(v-1)(\mu-\lambda)-2k}{\sqrt{\Delta}}\big]+1$$
otherwise. If $\theta_0\equiv1($mod $2)$, then
$$rank_2(MM^T)=\frac{1}{2}\big[(v+1)-\frac{(v-1)(\mu-\lambda)-2k}{\sqrt{\Delta}}\big].$$

$(2)$ $\theta_1\equiv1($mod $2)$, $\theta_2\equiv0($mod $2)$. If $\theta_0\equiv0($mod $2)$, then
$$rank_2(MM^T)=\frac{1}{2}\big[(v-1)+\frac{(v-1)(\mu-\lambda)-2k}{\sqrt{\Delta}}\big]$$
when $2|\mu$, or
$$rank_2(MM^T)=\frac{1}{2}\big[(v-1)+\frac{(v-1)(\mu-\lambda)-2k}{\sqrt{\Delta}}\big]+1$$
otherwise. If $\theta_0\equiv1($mod $2)$, then
$$rank_2(MM^T)=\frac{1}{2}\big[(v+1)+\frac{(v-1)(\mu-\lambda)-2k}{\sqrt{\Delta}}\big].$$

As $rank_2(M)\geq rank_2(MM^T)$, we have a lower bound on the $2$-rank of $M$ for certain choices $s,t,\lambda,\mu$.

\begin{thm}
If $\lambda-\mu\equiv1($mod $2)$, then
$$rank_2(M)\geq \min\big\{\frac{1}{2}\big[(v-1)-\frac{(v-1)(\mu-\lambda)-2k}{\sqrt{\Delta}}\big],\frac{1}{2}\big[(v-1)+\frac{(v-1)(\mu-\lambda)-2k}{\sqrt{\Delta}}\big]\big\}.$$
\end{thm}

\subsection{Code girth}
The girth of an LDPC code is also important in the performance.
Let $M$ be the incidence matrix of an SRPG$(s,t;\alpha_1,\alpha_2,\dots,\alpha_r;\lambda,\mu)$ $(\mathcal{P},\mathcal{B})$.
The \emph{Tanner graph} of the LDPC code $C(M)$ has vertex set $\mathcal{P}\cup\mathcal{B}$ with two vertices, $x$ and $y$, connected if and only if $x\in y$ or $y\in x$.
The \emph{girth} of $C(M)$ is the length of the shortest cycle of its Tanner graph.
Obviously, there exists a girth of six if and only if there is some $\alpha_i\geq2,1\leq i\leq r$. Similar to Lemma 5 in \cite{St}, we may have the following result.

\begin{thm}
The exact number of $6$-cycles in the Tanner graph of an LDPC code from an SRPG$(s,t;\alpha_1,\alpha_2,\dots,\alpha_r;\lambda,\mu)$ with some $\alpha_i\geq2,1\leq i\leq r$ is
$\frac{ns(s+1)(\lambda-s+1)}{6}.$
\end{thm}
\begin{proof}
Suppose $(\mathcal{P},\mathcal{B})$ is an SRPG$(s,t;\alpha_1,\alpha_2,\dots,\alpha_r;\lambda,\mu)$ with some $\alpha_i\geq2,1\leq i\leq r$. Take a block $B\in\mathcal{B}$ and a pair $\{P_1,P_2\}$ of points in $B$,
then $P_1$ is incident with $t$ blocks other than $B$, none of which contains $P_2$.
In these blocks, there are $\lambda-(s-1)$ points incident with $P_1,P_2$.
Then there are $\lambda-s+1$ $6$-cycles in which $P_1$ and $P_2$ are connected. Further,
 there are $\frac{s(s+1)(\lambda-s+1)}{2}$ $6$-cycles containing a pair of points in $B$ as there are $\binom{s+1}{2}$ pairs of points in $B$. Since there are $n$ blocks in total, and a single $6$-cycle includes three pairs of points, the result follows.
\end{proof}

\section{Performance}

In this section, we first discuss the $2$-ranks of $M_1$ and $M_2$, where $M_1,M_2$ are defined in Subsections 3.1 and 3.2, respectively.
Further, we show the performance of the LDPC codes.

\begin{thm}
$rank_2(M_1)=(q-1)^2.$
\end{thm}
\begin{proof}
Suppose $N_1=M_1M_1^T$, then $N_1$ has eigenvalues $\theta_0=(q-2)^2\equiv1($mod $2)$, $\theta_1=q\equiv1($mod $2)$,
$\theta_2=1\equiv1($mod $2)$ from Theorem \ref{ld11}.
So $rank_2(N_1)=(q-1)^2$ from Lemma \ref{23}.
Since $rank_2(M_1)\geq rank_2(N_1)$ and $M_1$ is a square matrix with order $(q-1)^2$, the $2$-rank of $M_1$ is $(q-1)^2$.
\end{proof}
\begin{thm}
$rank_2(M_2)\geq q^4-q^3-q^2+q.$
\end{thm}
\begin{proof}
Suppose $N_2=M_2M_2^T$, then $N_2$ has eigenvalues $\theta_0=q^2(q^2-1)\equiv0($mod $2)$, $\theta_1=q^3\equiv1($mod $2)$,
$\theta_2=q^2(q-1)\equiv0($mod $2)$ from Theorem \ref{ld12}.
Since $2|q(q-1)(q^2-q-1)$, $rank_2(N_2)=f_1=q^4-q^3-q^2+q$, where $f_1$ is the multiplicity of $\theta_1$ from Lemma \ref{23}.
So $rank_2(M_2)\geq rank_2(N_2)=q^4-q^3-q^2+q$.
\end{proof}
Using the software package MAGMA, we have known the $2$-ranks of $M_2$ are $81,625,$ $2401$ for $q=3,5,7$, respectively. From the  data, we give a conjecture on the $2$-rank of $M_2$.

\begin{conj}
$rank_2(M_2)=q^4$.
\end{conj}

We take a strongly regular generalized partial geometry constructed in Subsection 3.2, and compare the performance of LDPC code from it with that of randomly constructed code with the same code length in the additive white Gaussian noise (AWGN) channel, using the sum-product decoding algorithm from \cite{good}.
In Fig.1, the parity-check matrices of code from SRPG$(2,23;1,2,3;27,30)$ ($q=3$) in Subsection 3.2 and random code are $81\times648$ matrices.\vspace{0.1cm}

\begin{minipage}{0.4\linewidth}
{\small\textbf{Fig.1} Performance of an LDPC code in an AWGN channel using sum-product decoding.
A $(3,24)$-regular LDPC code from  SRPG$(2,23;$ $1,2,3;27,30)$ is compared with a randomly constructed $(3,24)$-regular LDPC code with the same length using a maximum of 1000 iterations. Their code rates are 0.875 and 0.878 respectively.}
\end{minipage}
\begin{minipage}{0.6\linewidth}
\centerline{\includegraphics[width=0.8\linewidth]{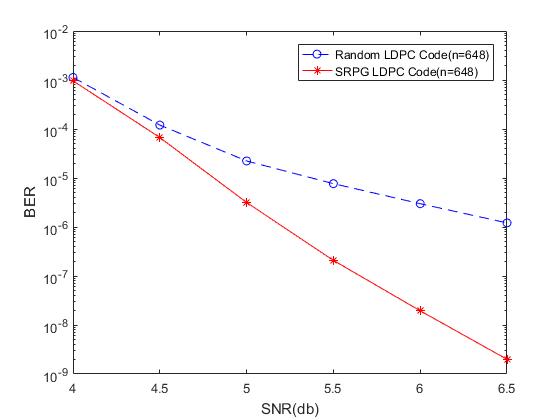}}\label{f1}
\end{minipage}
\vspace{0.1cm}

Fig.1 shows the performance of an LDPC code derived from the SRPG$(2,23;$ $1,2,3;$ $27,30)$. Compared with a randomly constructed LDPC code with the same length and weight, the BER of LDPC code from SRPG$(2,23;1,2,3;27,30)$ is lower than that of the randomly constructed code. This property shows that the LDPC code from strongly regular generalized partial geometry performs well.

\section{Conclusion}

In this paper, we introduce strongly regular generalized partial geometries of grade $r$ and present the related LDPC codes.
We construct two classes of strongly regular generalized partial geometries of grade $3$ based on quadrics in PG$(2,q)$ and PG$(3,q)$.
Besides,  we discuss some LDPC codes  by considering the combinatorial structures of strongly regular generalized partial geometries  and derive bounds on minimum distance, rank and girth for them. Further, LDPC codes from strongly regular generalized partial geometries offer improved error-correction performance over randomly constructed LDPC codes.
However, for the $2$-rank of $M_2$ defined in Subsections 3.2, we only give a conjecture, which needs further proof.\vspace{0.3cm}

\noindent {\bf Acknowledgements}

The authors wish to thank the referees for the hard work and the constructive comments.\vspace{0.3cm}

\noindent {\bf Declaration of competing interest}\vspace{0.3cm}

The authors declare that they have no known competing financial interests or personal relationships that could have
appeared to influence the work reported in this paper.\vspace{0.3cm}

\noindent {\bf Data availability}

No data was used for the research described in the paper.
\end{CJK*}

\end{document}